\documentclass[12pt]{amsart}
\usepackage{amscd,amsfonts,amsmath,amsthm,amssymb,mathrsfs,verbatim}
\usepackage{pstcol,pst-plot,pst-3d}%\usepackage[T1]{fontenc}
\usepackage{lmodern,pst-node, xfrac}

\def\NZQ{\Bbb}               % the font for N,Z,Q,R,C
\def\NN{{\NZQ N}}

\def\ZZ{{\NZQ Z}}
\def\RR{{\NZQ R}}

\def\AA{{\NZQ A}}

\def\frk{\frak}               % font for "Fraktur"

\def\Phi{{\frk n}}
\def\Phi{{\frk N}}
\def\MA{{\mathcal A}}

\def\MM{{\mathcal M}}

\def\MS{{\mathcal S}}

\def\MG{{\mathcal G}}
\def\ML{{\mathcal L}}

\def\opn#1#2{\def#1{\operatorname{#2}}} % to make operators
\opn\chara{char} \opn\length{\ell} \opn\pd{pd} \opn\rk{rk}
\opn\projdim{proj\,dim} \opn\injdim{inj\,dim} \opn\rank{rank}
\opn\depth{depth} \opn\grade{grade} \opn\height{height}
\opn\embdim{emb\,dim} \opn\codim{codim}

\opn\Tr{Tr} \opn\bigrank{big\,rank}
\opn\superheight{superheight}\opn\lcm{lcm}
\opn\trdeg{tr\,deg}%\emph{
\opn\reg{reg} \opn\lreg{lreg} \opn\ini{in} \opn\lpd{lpd}
\opn\size{size}\opn\bigsize{bigsize}
\opn\cosize{cosize}\opn\bigcosize{bigcosize}
\opn\sdepth{sdepth}\opn\sreg{sreg}
\opn\link{link}\opn\fdepth{fdepth}
\opn\div{div} \opn\Div{Div} \opn\cl{cl} \opn\Cl{Cl}
\opn\Spec{Spec} \opn\Supp{Supp} \opn\supp{supp} \opn\Sing{Sing}
\opn\Ass{Ass} \opn\Min{Min}\opn\Mon{Mon} \opn\dstab{dstab} \opn\astab{astab}
\opn\Ann{Ann} \opn\Rad{Rad} \opn\Soc{Soc}
\opn\Im{Im} \opn\Ker{Ker} \opn\Coker{Coker} \opn\Am{Am}
\opn\Hom{Hom} \opn\Tor{Tor} \opn\Ext{Ext} \opn\End{End}
\opn\Aut{Aut} \opn\id{id} \opn\span{span}

\opn\nat{nat}
\opn\pff{pf}%   \pf exists already
\opn\Pf{Pf} \opn\GL{GL} \opn\SL{SL} \opn\mod{mod} \opn\ord{ord}
\opn\Gin{Gin} \opn\Hilb{Hilb}\opn\sort{sort}
\opn\aff{aff} \opn\con{conv} \opn\relint{relint} \opn\st{st}
\opn\lk{lk} \opn\cn{cn} \opn\core{core} \opn\vol{vol}
\opn\link{link} \opn\star{star}\opn\lex{lex} \opn\Gr{Gr}
\opn\gr{gr}

\def\pot#1#2{#1[\kern-0.28ex[#2]\kern-0.28ex]}

\opn\dirlim{\underrightarrow{\lim}}
\opn\inivlim{\underleftarrow{\lim}}
\def\Implies{\ifmmode\Longrightarrow \else
        \unskip${}\Longrightarrow{}$\ignorespaces\fi}
\def\implies{\ifmmode\Rightarrow \else
        \unskip${}\Rightarrow{}$\ignorespaces\fi}
\def\iff{\ifmmode\Longleftrightarrow \else
        \unskip${}\Longleftrightarrow{}$\ignorespaces\fi}

\let\:=\colon
\newtheorem{Theorem}{Theorem}[section]
\newtheorem{Lemma}[Theorem]{Lemma}

\newtheorem{Proposition}[Theorem]{Proposition}
\newtheorem{Remark}[Theorem]{Remark}

\newtheorem{Example}[Theorem]{Example}
\newtheorem{Examples}[Theorem]{Examples}

\newtheorem{Question}[Theorem]{Question}

\let\epsilon\varepsilon
\let\kappa=\varkappa
\def\qed{\ifhmode\textqed\fi
      \ifmmode\ifinner\quad\qedsymbol\else\dispqed\fi\fi}
\def\textqed{\unskip\nobreak\penalty50
       \hskip2em\hbox{}\nobreak\hfil\qedsymbol
       \parfillskip=0pt \finalhyphendemerits=0}
\def\dispqed{\rlap{\qquad\qedsymbol}}

\opn\dis{dis}
\def\pnt{{\raise0.5mm\hbox{\large\bf.}}}

\opn\Lex{Lex}

\begin{document}

%\begin{frontmatter}

\title{ Markov complexity of monomial curves }
\author{Hara Charalambous, Apostolos Thoma, Marius Vladoiu}\thanks{The third author was partially supported from grant PN--II--ID--PCE--2011--3--1023, nr. 247/2011, awarded by UEFISCDI}
 
\address{Hara Charalambous, Department of Mathematics, Aristotle University of Thessaloniki, Thessaloniki 54124, Greece }
\email{hara@math.auth.gr}

\address{Apostolos Thoma, Department of Mathematics, University of Ioannina, Ioannina 45110, Greece } 
\email{athoma@uoi.gr}
%\corauth[cor1]{Corresponding author}

\address{Marius Vladoiu, Faculty of Mathematics and Computer Science, University of Bucharest, Str. Academiei 14, Bucharest, RO-010014, Romania, and}
\address{Simion Stoilow Institute of Mathematics of Romanian Academy, Research group of the project ID--PCE--2011--3--1023, P.O.Box 1--764, Bucharest
014700, Romania}
\email{vladoiu@gta.math.unibuc.ro}


\begin{abstract}
\par Let $\mathcal{A}=\{{\bf a}_1,\ldots,{\bf a}_n\}\subset\Bbb{N}^m$. We give an algebraic characterization of the universal Markov basis of the toric ideal $I_{\mathcal{A}}$. We show that the Markov complexity of $\mathcal{A}=\{n_1,n_2,n_3\}$ is equal to two if $I_{\mathcal{A}}$ is complete intersection and equal to three otherwise, answering a question posed by Santos and Sturmfels. We prove that for any  $r\geq 2$ there is a unique minimal Markov basis of $\mathcal{A}^{(r)}$. Moreover, we prove that for any integer $l$ there exist integers $n_1,n_2,n_3$ such that the Graver complexity of $\mathcal{A}$ is greater than $l$.  
\end{abstract}
\maketitle
%\end{frontmatter}

\section*{Introduction}
Let $\Bbbk$ be a field, $n,m\in \NN$,  $\MA=\{{\bf a}_1,\ldots,{\bf a}_n\}\subset\NN^m$ and $A\in\MM_{m\times n}(\NN)$ be the matrix whose columns are the vectors of $\MA$. We  let $\ML({\MA}):=\Ker_{\ZZ}(A)$  be the corresponding sublattice of $\ZZ^n$  and denote by $I_{\MA}$ the corresponding  toric ideal of $\MA$ in $\Bbbk[x_1,\ldots, x_n]$. We recall that 
$I_{\MA}$ is generated by all binomials of the form $x^{\bf u }-x^{\bf w}$ where
$\ {\bf u}-{\bf w} \in  \mathcal{L}(\MA)$.

A {\em Markov basis} of $\MA$  is a finite subset $\MM$ of $\mathcal{L}(\MA)$ such that whenever ${\bf w}, {\bf u}\in \NN^{n}$ and ${\bf w}- {\bf u}\in \mathcal{L}(\MA)$ (i.e. $A {\bf w}=A {\bf u}$), there exists a subset $\{{\bf v}_i: i=1,\ldots, s\}$ of  ${\MM}$ that {\it connects} ${\bf w}$ to ${\bf u}$. This means that   $({\bf w}- {{\sum^p_{i=1}}}{{\bf v}_i})\in \NN^n$  for all $1\leq p\leq s$ and   ${\bf w}- {\bf u}=\sum^s_{i=1}{\bf v}_i $. A Markov basis $\MM$ of $\MA$  is {\it minimal} if no subset of $\MM$ is a Markov basis of $\MA$. For a vector ${\bf u}\in \mathcal{L}(\MA)$ we let $\bf u^+$, ${\bf u}^-$ be the unique vectors in $ \NN^{n}$ such that  $\bf u= \bf u^+ -\bf u^-$. If $\MM$  is a minimal Markov basis of $\MA$ then a classical result of Diaconis and Sturmfels  states that the set $\{x^{\bf u^+}-x^{\bf u^-}: \ {\bf u}\in \MM\}$ is a   minimal generating set of $I_{\MA}$, see \cite[Theorem 3.1]{DS}.  The {\em universal Markov basis} of $\MA$, which we denote by ${\MM}(\MA)$, is the union of all minimal Markov bases of $\MA$, where we  identify  elements that differ by a sign, see \cite[Definition 3.1]{HS}.  The intersection of all minimal Markov bases of $\MA$ via the same identification, is called the {\it indispensable subset} of the universal Markov basis $\MM(\MA)$ and  is denoted by $\MS(\MA)$.  The {\it Graver basis} of $\MA$,  $\MG(\MA)$,  is the subset of $\mathcal{L}(\MA)$ whose elements have {  no} {\it  proper conformal decomposition}, i.e.~${\bf u}\in \mathcal{L}(\MA)$ is in $\MG(\MA)$ if there is no other ${\bf v}\in \mathcal{L}(\MA)$ such that ${\bf v^+}\leq {\bf u^+}$ and   ${\bf v^-}\leq {\bf u^-}$, see \cite[Section 4]{St}.
The {\it Graver basis} of $\MA$ is always a finite set and contains the universal Markov basis of $\MA$,  see \cite[Section  7]{St}. Thus  the following inclusions hold: 
 \[\MS(\MA)\subseteq\MM(\MA)\subseteq\MG(\MA) .\] 
In \cite{CKT} a description was given for the elements of $\MS(\MA)$ and $\MM(\MA)$ that had a geometrical flavour: it considered the various {\it fibers} of $\MA$ in $\NN^n$ and the connected components of certain graphs. It did not examine the problem from a strict algebraic point of view such as  conformality.  This  point of view is   seen in \cite{HS}, but only for the elements of $\MS(\MA)$ from the side of sufficiency. In \cite{HS}, the authors show that any  vector with  no {\it proper semiconformal decomposition} is necessarily in  $\MS(\MA)$, see \cite[Lemma 3.10]{HS}. 
In  this paper we attempt to give the complete algebraic characterization for  the elements of $\MS(\MA)$ and $\MM(\MA)$.  This is done in Section 1.  In Proposition~\ref{indispensabledec} we prove that the condition of \cite[Lemma 3.10]{HS} is not only sufficient but also necessary. Next, to give the algebraic characterization of the vectors in $\MM(\MA)$, we introduce the notion of a {\it  proper strongly semiconformal decomposition}   and prove that the nonzero vectors with   {  no} proper strongly semiconformal decomposition are precisely the vectors of $\MM(\MA)$,  see Proposition~\ref{universalmarkov}. The relationship between these decompositions is given in Lemma~\ref{c_sc_ssc}. Schematically the following implications hold:
\[ \textrm{ proper conformal } \Rightarrow \textrm { proper strongly semiconformal }  \Rightarrow \textrm { proper semiconformal }   . \]
In Example~\ref{non_c_sc_ssc} we show that these implications are the best one could hope. It is important to note that the definitions of conformal and semiconformal decompositions involve exactly two summands. The natural and easy generalization to   decompositions involving   $l$ summands, $l\ge 2$, does not produce anything new: such decompositions lead  to  a conformal and semiconformal   decomposition with exactly two summands. This fact stands in contrast to the definition of a strongly semiconformal decomposition. As is shown in Example~\ref{2_not_enough}, a vector may have a  proper strongly semiconformal decomposition into $l$ vectors with $l>2$, but {\it not} a proper strongly semiconformal decomposition into exactly $2$ vectors. There are however certain classes of integer configurations for which the notion of proper strongly semiconformality into $l$ vectors with $l\ge 2$ coincides with the notion of  proper strongly semiconformality  into 2 vectors. Such is the   class given by the monomial curves in $\AA^3$ as we show in  Lemmas~\ref{uniquesemisum} and~\ref{sccompint}. Another class is given by the Lawrence liftings of  monomial curves in $\AA^3$, as follows from  Theorems~\ref{markovnotci} and~\ref{markcompci}. We also note that  incidence matrices of  graphs have this property, see \cite[Propositions 4.3 and 4.8]{RTT}. 

For $A\in\MM_{m\times n}(\NN)$ as above and  $r\ge 2$, the $r$--th {\it Lawrence lifting} of $A$   is denoted by $A^{(r)}$ and is the $(rm+n)\times rn$ matrix 

\[
A^{(r)}=\begin{array}{c} \overbrace{\quad\quad\quad \quad \quad\ \ }^{r-\textrm{times}} \\
 \left( \begin{array}{cccc}
\ A\ & 0 &   & 0 \\
0 &\ A\ &  & 0 \\
 & & \ddots &  \\
0 & 0 &  &\ A\ \\
I_n &I_n& \cdots & I_n
\end{array} \right)\end{array},\]
see \cite{SS}. We write $\ML(\MA^{(r)})$ for $\Ker_{\ZZ}(A^{(r)})$, denote by $\MA^{(r)}$ the matrix $A^{(r)}$, and  identify an element of $\mathcal{L}(\MA^{(r)})$ with  an $r\times n $ matrix:   each row of this matrix corresponds to an element of $\mathcal{L} (\MA)$ and the sum of its rows is zero. The   {\em type} of an element of $\mathcal{L}(\MA^{(r)})$ 
 is the number of  nonzero rows of this matrix. The {\em Markov complexity}, $m(\MA)$, is the largest type of any vector in the universal Markov basis of $A^{(r)}$ as $r$ varies. The {\em Graver complexity} of $\MA$,  $g(\MA)$, is the largest type of any vector in the Graver basis of $A^{(r)}$, as $r$ varies. We note that the study of $A^{(r)}$, for $A\in\MM_{m\times n}(\NN)$ was motivated by consideration of hierarchical models in Algebraic Statistics, see \cite{SS}. Aoki and Takemura, in \cite{AT}, while studying Markov bases for certain contingency tables with zero two-way marginal totals, gave the first examples of matrices with finite Markov complexity, see \cite[Theorem 4]{AT}.  In \cite[Theorem 1]{SS}, Santos and Sturmfels proved that  $m(\MA)$   is  bounded above by the Graver complexity of $\MA$, $g(\MA)$, and since the latter one is finite, $m(\MA)$ is also finite. In fact, $g(\MA)$ is the maximum $1$-norm of any element in the Graver basis of the Graver basis of $\MA$, \cite[Theorem 3]{SS}. Up to now, no formula for  $m(\MA)$ is known in general and there are only a few classes of toric ideals for which $m(\MA)$ has been computed, see \cite{AHT, AT, SS}. In this paper we compute   $m(\MA)$, when $\MA$ is a monomial curve in $\AA^3$.  This answers a question posed by Santos and Sturmfels in \cite{SS}, see Example 6 of that paper. We succeed in computing $m(\MA)$ by applying the results  of Section 1. 

The topic of  monomial curves has been the subject of extensive research ever since Herzog in \cite{He} studied such configurations. We recall that a monomial curve in the $d$-dimensional affine space $\AA^d$ is defined as the curve $\{(t^{n_1},\ldots,t^{n_d}):\ t\in \Bbbk\}$, where $n_1,\ldots,n_d$ are positive integers such that $\gcd(n_1,\ldots,n_d)=1$.  In this paper we write $\MA=\{ n_1,n_2,n_3\}\subset \ZZ_{>0}$ and implicitly assume that $\gcd(n_1,n_2,n_d)=1$. In \cite[Theorem 3.8]{He}, it was shown that  the toric ideal $I_{\MA}$ is  either a complete intersection or if not, then it is minimally generated by exactly three binomials.  We deal with each case separately. In Section 2, we consider the case when $I_{\MA}$ is {\it not} a complete intersection. We show that the universal Markov basis of $\MA^{(r)}$ is unique and compute its cardinality, see Theorem~\ref{markovnotci}. We also show   in the  same Theorem~\ref{markovnotci}  that $m(\MA)=3$. In Section 3, we consider the case when $I_{\MA}$ is  a complete intersection. As before we show that the universal Markov basis of $\MA^{(r)}$ is unique, compute its cardinality, and show that $m(\MA)=3$, see Theorem~\ref{markcompci}. We note that Herzog in \cite{He} describes all possible minimal generating sets of $I_{\MA}$ in either case,     an essential tool to our study. To be more precise,   with the notation of \cite{He}, for  $i\in\{1,2,3\}$ we consider $c_i$ to be the smallest element of $ \ZZ_{>0}$ such that there exist integers  $r_{ij},r_{ik}\in \NN$ with $\{i,j,k\}=\{1,2,3\}$, and with the property that $c_i n_i=r_{ij}n_j+ r_{ik}n_k$. What determines whether $I_{\MA}$ is a complete intersection or not is whether there are $i,j\in\{1,2,3\}$ such that $r_{ij}=0$. If $r_{ij}>0$ for all $i,j=1,2,3$ then $I_{\MA}$ is minimally generated by exactly three binomials and in this case, $I_{\MA}$ has a  unique minimal generating set which is explicitly described in \cite[Proposition 3.2, Proposition 3.3]{He}. If there exist $i,j\in\{1,2,3\}$ such that $r_{ij}=0$ then $I_{\MA}$ is a complete intersection and has no unique minimal binomial generating set. In this case the   universal Markov basis of $\MA$ is  explicitly described in \cite[Proposition 3.5]{He}.

As mentioned above, in Theorems~\ref{markovnotci} and~\ref{markcompci}, the uniqueness of the minimal Markov basis of $\MA^{(r)}$  for $r\ge 2$  was proved  for all monomial curves $\MA$ in $\AA^3$.  We want to further dwell on  this fact. It is well known that  $\MA^{(2)}$ has a unique minimal Markov basis for all $\MA\subset\NN^m$, see \cite[Theorem 7.1]{St}.  Moroever for  $r\geq 3$, the uniqueness of the minimal Markov bases for $\MA^{(r)}$ is noted for some classes of toric ideals, see \cite{AT,OH,SS}, see also     related \cite[Conjecture 3.8]{OH}. It is thus of interest to note that Lawrence liftings    of  monomial curves   in $\AA^3$   have   unique minimal Markov basis, although the curves themselves may not. 

In Section 4 we compare $m(\MA)$ with $g(\MA)$ when $\MA=\{n_1,n_2,n_3\} $ is a monomial curve in $\AA^3$. To find $g(\MA)$ we have to compute the Graver basis of the Graver basis of $\MA$, which is difficult since there is no good description of the Graver bases of monomial curves in $\AA^3$. In Theorem~\ref{lowerboundgraver} we  show that there is no integer that bounds from above the Graver complexities of all monomial curves and thus the difference between $m(\MA)$ and $g(\MA)$ can be made arbitrarily large. Furthermore we pose some questions that were motivated and partly based on the extensive computations that were done with 4ti2,  \cite{4ti2}.

\section{ Universal Markov basis}

Let $\MA=\{{\bf a}_1,\ldots,{\bf a}_n\}\subset \NN^m$ and 
 $A\in\MM_{m\times n}(\NN)$,   $\ML({\MA})\subset \ZZ^n$, $I_{\MA}\subset \Bbbk[x_1,\ldots, x_n]$
be the corresponding matrix, lattice and toric ideal of  $\MA$ respectively. We  note that the only invertible element of   
$\NN \MA$ is ${\bf 0}$ and   $\ML({\MA})\cap\NN^n=\{\bf 0\}$. 

Let $\bf u$, ${\bf w}_1,{\bf w}_2 \in \ML({\MA})$ be such that ${\bf u}={\bf w}_1 + {\bf w}_2$. We say that the above sum is a {\bf conformal decomposition} of ${\bf u}$ and write ${\bf u}={\bf w}_1 +_{c} {\bf w}_2$ if  ${\bf u}^+={\bf w}_1^+ + {\bf w}_2^+$ and ${\bf u}^-={\bf w}_1^- + {\bf w}_2^-$.  If both ${\bf w}_1,{\bf w}_2$ are nonzero, we call such a decomposition   {\bf  proper}. The  {\em Graver basis} of $\MA$, $\MG(\MA)$,  consists of the nonzero vectors in $\ML({\MA})$ for which there is no proper conformal decomposition and is a finite set, see for example \cite[Algorithm 7.2]{St}. 

The notion of a semiconformal decomposition was introduced in \cite[Definition 3.9]{HS}. Let ${\bf u},{\bf v},{\bf w}\in\ML({\MA})$. We say that  ${\bf u}={\bf v}+_{sc}{\bf w}$  is a {\bf semiconformal decomposition}   of ${\bf u}$ if ${\bf u}={\bf v}+{\bf w}$ and ${\bf v}(i)>0$ implies that ${\bf w}(i)\geq 0$ and ${\bf w}(i)<0$ implies that ${\bf v}(i)\leq 0$ for  $1\leq i\leq n$.  Here ${\bf v}(i)$ denotes the $i^{th}$ coordinate of the vector ${\bf v}$. We call the decomposition {\bf proper} if both ${\bf v}, {\bf w}$ are nonzero. It is easy to see that  ${\bf u}={\bf v}+_{sc} {\bf w}$  if and only if  ${\bf u}^+\geq {\bf v}^+$ and ${\bf u}^-\geq {\bf w}^-$. We remark that  $\bf 0$ cannot be written as the semiconformal sum of two nonzero vectors  since $\ML({\MA})\cap \NN^n=\{\bf 0\}$. When writing a semiconformal decomposition of ${\bf u}$ it is necessary to specify the order of the vectors added. A semiconformal decomposition of ${\bf u}$ for which the order of the vectors can be reversed is a conformal decomposition, that is
\[ \textrm{if }
{\bf u}={\bf v}+_{sc} {\bf w} \textrm{ and } {\bf u}={\bf w}+_{sc} {\bf v}  \textrm{ then } {\bf u}={\bf v}+_{c} {\bf w}.
\] 
We note that a semiconformal decomposition of ${\bf u}$ gives rise to a semiconformal decomposition of $-{\bf u}$ and vice versa, by simply reversing the order of the summands:
\[
{\bf u}={\bf v}+_{sc} {\bf w} \Leftrightarrow -{\bf u}=(-{\bf w})+_{sc} (-{\bf v})\ .
\]
Let ${\bf u}\in\ML(\MA)$. The {\em fiber} $\mathcal{F}_{{\bf u}}$  is the set $\{{\bf t}\in\NN^n : {{\bf u}}^+-{\bf t}\in \ML(\MA)\}$. $\mathcal{F}_{{\bf u}}$ is  a finite set, see for example \cite[Proposition 2.3]{CTV}. Next we show that lack of a proper semiconformal decomposition  is not only a sufficient condition for an element to be in ${\mathcal S}(\MA)$ as   was shown in  \cite[Lemma 3.10]{HS}, but it is also 
a necessary  condition.  

\begin{Proposition}\label{indispensabledec}  The indispensable part ${\mathcal S}(\MA)$ of the universal Markov basis consists of all nonzero vectors in $\ML({\MA})$ which have no proper semiconformal decomposition.
\end{Proposition}
\begin{proof} We only need to show that if ${\bf u}\in {\mathcal S}(\MA)$ then ${\bf u}$ has no proper semiconformal decomposition. Suppose that ${\bf u}={\bf v}+_{sc}{\bf w}$ for some nonzero vectors ${\bf v,w}\in\ML(\MA)$. Since ${\bf u}\in\MS(\MA)$ the binomial  $x^{{\bf u}^+}-x^{{\bf u}^-}$ belongs to all minimal system of generators of $I_{\MA}$. It follows that $\mathcal{F}_{\bf u}$ consists of  exactly two elements, ${\bf u}^+$ and ${\bf u}^-$, see \cite[Corollary 2.10]{CKT}. On the other hand ${\bf u}={\bf v}+_{sc} {\bf w}$ implies that  ${\bf u}^{+} -{\bf v}= {\bf u}^{-}+{\bf w}\in \NN^n$. Moreover $A ({\bf u}^{+} -{\bf v})=A {\bf u}^{+}$ and thus ${\bf u}^{+} -{\bf v}$ is in the fiber of ${\bf u}^+$. This implies that either ${\bf u}^{+} -{\bf v}={\bf u}^{+}$ and thus ${\bf v}={\bf 0}$ or ${\bf u}^{+} -{\bf v}={\bf u}^{-}$ which then  implies that ${\bf w}={\bf 0}$, a contradiction in either case.
\end{proof}

Let ${\bf u}, {\bf u}_1,\ldots,{\bf u}_l \in\ML({\MA})$, $l\geq 2$. We say that ${\bf u}=_{ssc}{\bf u}_1+\cdots+{\bf u}_l$,  is a {\bf strongly semiconformal decomposition} if  ${\bf u}={\bf u}_1+\cdots+{\bf u}_l$ and the following conditions are satisfied:
\[
{\bf u}^+>{\bf u}_1^+ \quad \text{ and } \quad {\bf u}^+ >(\sum_{j=1}^{i-1} {\bf u}_j)+{\bf u}_i^+ \ \text{ for all } \ i=2,\ldots,l.
\] 
When $l=2$, we simply write  ${\bf u}={\bf u}_1+_{ssc} {\bf u}_2$. Note that ${\bf u}={\bf u}_1+_{ssc} {\bf u}_2$ implies that ${\bf u}^+>{\bf u}_1^+$ and ${\bf u}^->{\bf u}_2^-$.
We say that the decomposition is {\bf proper} if all ${\bf u}_1,\ldots,{\bf u}_l$ are nonzero. 
We remark that if 
${\bf u}=_{ssc}{\bf u}_1+\cdots+{\bf u}_l$   is proper then
 ${\bf u}^+,{\bf u}^+-{\bf u}_1,\ldots,{\bf u}^+-\sum_{i=1}^l {\bf u}_i={\bf u}^-\in \NN^n$ and thus are distinct elements of 
 $ \mathcal{F}_{{\bf u}}$. 
In the following lemma we show  the implications amongst the  three types of decompositions we defined above. It is immediate that a conformal decomposition is also a semiconformal decomposition. 

\begin{Lemma}\label{c_sc_ssc}
Let ${\bf u}\in\ML({\MA})$ be a nonzero vector. Then the following hold:
\begin{enumerate}
\item[(i)] If ${\bf u}$ has a proper conformal decomposition then ${\bf u}$ has a proper strongly semiconformal decomposition,
\item[(ii)] If ${\bf u}$ has a proper strongly semiconformal decomposition then ${\bf u}$ has a proper semiconformal decomposition.
\end{enumerate}
\end{Lemma}
\begin{proof}
For (i) note that if ${\bf u}={\bf v}+_{c}{\bf w}$ then ${\bf u}^+={\bf v}^++{\bf w}^+$ and ${\bf u}^-={\bf v}^-+{\bf w}^-$. If ${\bf u}^+={\bf v}^+$ then ${\bf w}^+={\bf 0}$ and thus ${\bf w}={\bf 0}$, a contradiction. Similarly, one shows that ${\bf u}^-\neq{\bf w}^-$ and thus ${\bf u}={\bf v}+_{ssc}{\bf w}$.  

In order to prove (ii), assume that ${\bf u}$ admits a proper strongly semiconformal decomposition: there exists 
$l\geq 2$ and ${\bf u}_1,\ldots,{\bf u}_l\in\ML({\MA})\setminus\{\bf 0\}$ such that ${\bf u}=_{ssc}{\bf u}_1+\cdots+{\bf u}_l$. In particular,  ${\bf u}^+>{\bf u}_1^+$. We let ${\bf v}={\bf u}_2+\cdots+{\bf u}_l$.
We will show  that ${\bf u}={\bf u}_1+_{sc}{\bf v}$. Indeed, since $A({\bf u}^+-{\bf u}_1)=A{\bf u}^+$ and ${\bf u}^+-{\bf u}_1={\bf u}^+-{\bf u}_1^++{\bf u}_1^->{\bf 0}$ it follows that ${\bf u}^+-{\bf u}_1$ belongs to the fiber of ${\bf u}^+$. Moreover, ${\bf u}^+\neq{\bf u}^+-{\bf u}_1$ since ${\bf u}_1\neq {\bf 0}$. On the other hand, if ${\bf u}^+-{\bf u}_1={\bf u}^-$ then we obtain that ${\bf u}={\bf u}_1$ and thus ${\bf u}^+={\bf u}_1^+$, a contradiction. Therefore, the fiber $\mathcal{F}_{{\bf u}}$ contains at least three different vectors: ${\bf u}^+,{\bf u}^+-{\bf u}_1,{\bf u}^-$. Hence we have the following expression
\[
x^{{\bf u}^+}-x^{{\bf u}^-}=x^{{\bf u}^+-{\bf u}_1^+}(x^{{\bf u}_1^+}-x^{{\bf u}_1^-})+x^{\bf t}(x^{{\bf u}^+-{\bf u}_1-{\bf t}}-x^{{\bf u}^--{\bf t}}),
\] 
where $x^{\bf t}$ is the monomial with the property that $\gcd(x^{{\bf u}^+-{\bf u}_1},x^{{\bf u}^-})=x^{\bf t}$. The last condition implies that there exists a nonzero vector ${\bf w}\in\ML({\MA})$ such that ${\bf w}^+={\bf u}^+-{\bf u}_1-{\bf t}$ and ${\bf w}^-={\bf u}^--{\bf t}$. It is easy to see that ${\bf u}={\bf u}_1+_{sc} {\bf w}$. On the other hand we have ${\bf u}={\bf u}_1+{\bf v}$ and thus ${\bf v}={\bf w}$, which implies our claim.
\end{proof}

The reverse implications from Lemma~\ref{c_sc_ssc} do not follow as the next example shows. 

\begin{Example}\label{non_c_sc_ssc}
{\em Let $\MA=\{2,3,11\}$. The  toric ideal $I_{\MA}$ is a complete intersection and there are exactly two minimal Markov bases of $\MA$: $\{(-3,2,0), (4,1,-1)\}$ and $\{(-3,2,0),(1,3,-1)\}$, see Proposition~\ref{completeint}. Thus $\mathcal M(\MA)=\{(-3,2,0)$, $(4,1,-1)$, $ (1,3,-1)\}$ while ${\mathcal S}(\MA)=\{(-3,2,0)\}$. Moreover, computation with 4ti2 \cite{4ti2} show that the elements of  $\mathcal{G}(\MA)$  are  ${\bf u}_1= (0,11,-3)$, 
${\bf u}_2=(3,-2,0)$, ${\bf u}_3=(4,1,-1)$, ${\bf u}_4=(1,3,-1)$, ${\bf u}_5=(7,-1,-1)$, ${\bf u}_6=(11,0,-2)$, ${\bf u}_7= (1,-8,2)$, ${\bf u}_8=(2,-5,1) $. We note that ${\bf u}_5={\bf u}_2+_{ssc} {\bf u}_3$ while $ {\bf u}_5\neq {\bf u}_2+_{c} {\bf u}_3 $ and that even though
${\bf u}_3={\bf u}_2+_{sc} {\bf u}_4$ the strong semiconformality does not hold:  ${\bf u}_3\neq {\bf u}_2+_{ssc} {\bf u}_4$. 
}
\end{Example}

If ${\bf u}\in \ML({\MA})$ we associate an  $\MA$-degree to ${\bf u}$ and  the binomial $B=x^{{{\bf u}}^+}-x^{{{\bf u}^-}}\in I_{\MA}$ as follows: 
\[ 
\deg_{\MA}( {\bf u})=\deg_{\MA}(B)=\sum_{i=1}^n {\bf u}^+(i)\;{\bf a}_i.
\]
We note that by  \cite[Proposition 2.2]{CKT} ${\bf u}$ is in a minimal Markov basis  of  ${\MA}$ and $x^{{{\bf u}}^+}-x^{{{\bf u}^-}}$ is part of a minimal generating set of $I_{\MA}$ if and only if $x^{{{\bf u}}^+}-x^{{{\bf u}^-}}$ is not in the ideal generated by the binomials of $I_{\MA}$ of strictly smaller $\MA$-degrees.

\begin{Proposition}\label{universalmarkov} The universal Markov basis ${\MM}(\MA)$ of $\MA$ consists  of all nonzero vectors in $\ML({\MA})$ with no proper strongly  semiconformal decomposition.
\end{Proposition}
\begin{proof}
 Let ${\bf u}\in {\MM}(\MA)$ and suppose that ${\bf u}=_{ssc} {\bf u}_1+\cdots +{\bf u}_l$. We consider the binomials $B=x^{{\bf u}^{+}}-x^{{\bf u}^-}$ and $B_i=x^{{\bf u}_i^+}-x^{{\bf u}_i^-}$ for $i=1,\ldots, l$. It is immediate that
\[
B= x^{{\bf u}^+-{\bf u}_1^+} B_1+x^{{\bf u}^+-{{\bf u}_1}-{\bf u}_2^+} B_2+ \cdots + x^{{\bf u}^+-\sum_{i=1}^{l-1} {\bf u}_i-{\bf u}_l^+}B_l.
\]
The  strongly semiconformality assumption implies that the coefficients in the above expression are all non-constant monomials. Since $B_i\in I_{\MA}$ for $i=1,\ldots, l$ it follows that $B$ is in the ideal generated by the binomials of $I_{\MA}$ of strictly smaller $\MA$-degree than $B$. By \cite[Section 1.3]{DSS}, $B$ can not be part of any minimal system of generators of $I_{A}$, a contradiction.

Suppose now that ${\bf u}\in \ML({\MA})\setminus\{\bf 0\}$ and that ${\bf u}\not\in\MM(\MA)$. It follows that $B=x^{{ {\bf u}}^{+}}-x^{{{\bf u}}^-}$ is in the ideal generated by the binomials of $I_{\MA}$ of strictly smaller $\MA$-degree than $B$.  By \cite[Proposition 3.11]{CTV},  there are monomials $x^{{\bf t}_i}\neq 1$ and binomials $B_i=x^{{{\bf u}_i}^+}- x^{{ {\bf u}_i}^-}$ for $i=1,\ldots, l$ where ${\bf u}_i\in \ML({\MA})$  such that
\[ 
B=x^{{\bf t}_1} B_1+\cdots+x^{{\bf t}_l} B_l\ ,
\]
\[
x^{{\bf t}_1}x^{\bf u_1^+}=x^{\bf u^+},\;\; x^{{\bf t}_i}x^{\bf u_i^+}=x^{{\bf t}_{i-1}} x^{\bf u_{i-1}^-}, {\text{ for }} \; i=2,\ldots, l, \textrm{ and } x^{{\bf t}_l} x^{\bf u_l^-}=x^{\bf u^-}\ . 
\] 
Note that the binomials $B_i$ in this expression need not be distinct. Thus ${\bf u}={\bf u}_1+\cdots+{\bf u}_l$ and
\[
{\bf t}_1={\bf u}^+-{\bf u}_1^+, \ {\bf t}_i= {\bf u}^+ -(\sum_{j=1}^{i-1} {\bf u}_j)-{\bf u}_i^+,\; i=2,\ldots, l \ .
\]
Since  ${\bf t}_i>{\bf 0}$ for $i=1,\ldots, l$ it follows that ${\bf u}=_{ssc}{\bf u}_1+{\bf u}_2+\cdots+{\bf u}_l$.
\end{proof}

The following example shows that it is necessary to define the strongly semiconformality in 
terms of $l$ vectors where $l\ge 2$. 

\begin{Example}\label{2_not_enough}
{\em Let $\MA=\{(2,2),(2,0),(0,2),(1,3),(3,3)\}$ be a subset of $\NN^2$. The corresponding matrix 
$A\in\MM_{2\times 5}(\NN)$ is
\[
A=\left( \begin{array}{ccccc}
2 & 0 & 2 & 1 & 3 \\
2 & 2 & 0 & 3 & 3 \\
\end{array} \right).
\]
Using 4ti2 \cite{4ti2} and \cite[Theorem 2.6]{CKT} we get that  
\[
\MM(\MA)=\{(1,-1,-1,0,0), (0,0,1,1,-1), (0,3,1,-2,0), (1,2,0,-2,0)\}.
\]
We consider  the vector ${\bf u}=(2,1,0,-1,-1)\in \ML(\MA)$. We note that ${\bf u}^+=(2,1,0,0,0)$,
${\bf u}^-=(0,0,0,1,1)$ and $\deg_{\MA}({\bf u})=(4,6)$. 
It can be easily seen that 
\[
\mathcal F_{{\bf u}}=\{(2,1,0,0,0),(1,2,1,0,0),(0,3,2,0,0),(0,0,1,2,0),(0,0,0,1,1)\}.
\]
We denote the elements of $\mathcal F_{\bf u}$  by ${\bf v}_1,\ldots,{\bf v}_5$ written in the above order.     It is straightforward that
\[
{\bf u}=_{ssc} {\bf u}_1+{\bf u}_2+{\bf u}_3,
\]   
where ${\bf u}_1={\bf v}_1-{\bf v}_2$, ${\bf u}_2={\bf v}_2-{\bf v}_4$, ${\bf u}_3={\bf v}_4-{\bf v}_5$ $\in\ML(\MA)$. Thus according to Theorem~\ref{universalmarkov}, ${\bf u}\notin \MM(\MA)$. It is interesting to note that  
${\bf u}$ does not have a proper strongly semiconformal decomposition into two vectors. Indeed, if ${\bf u}={\bf w}_1+_{ssc}{\bf w}_2$ then   $ {\bf u}^+-{\bf w}_1={\bf u}^-+{\bf w}_2\in \mathcal F_{{\bf u}}$ and cannot equal ${\bf v}_1= {\bf u}^+$ nor $ {\bf v}_5={\bf u}^-$. Thus $ {\bf u}^+-{\bf w}_1$ is
either ${\bf v}_2,{\bf v}_3$ or ${\bf v}_4$. This implies that ${\bf w}_1={\bf v}_1-{\bf v}_i$ and ${\bf w}_2={\bf v}_i-{\bf v}_5$ for an integer $i\in\{2,3,4\}$. But for $i\in\{2,3\}$ we have ${\bf w}_2^-={\bf v}_5={\bf u}^-$, while if $i=4$ then ${\bf w}_1^+={\bf v}_1={\bf u}^+$, which leads to a contradiction in either case since ${\bf u}={\bf w}_1+_{ssc}{\bf w}_2$. 
}
\end{Example} 
\medskip

Next, for $r\ge 2$ we consider the  Lawrence lifting $\MA^{(r)}$ of $\MA$.  The following remark allows us to restrict the discussion of  semiconformal decompositions  of elements of  $\ML(\MA^{(r)})$ to  semiconformal  decompositions  of elements of maximum  type $r$.

\begin{Remark}\label{declaw}
{\em Let $r\geq 2$, ${\bf v}\in\ML(\MA^{(r)})$  be a nonzero vector of type $s< r$ and  
without loss of generality assume that the nonzero rows of ${\bf v}$ are the first $s$ rows. Note that $s\geq 2$. We define ${\bf w}\in\ML(\MA^{(s)})$ as the vector whose $i^{th}$ row is just the $i^{th}$ row of ${\bf v}$ for   $i=1,\ldots,s$. Then ${\bf v}$ admits a proper semiconformal decomposition if and only if ${\bf w}\in\ML(\MA^{(s)})$ admits a proper semiconformal decomposition. The equivalence is immediate since the unique semiconformal decomposition of ${\bf 0}$ is into zero vectors. The above discussion holds for proper strongly semiconformal decompositions ${\bf v}$ as well. }
\end{Remark}

\noindent Let ${\bf v}=({\bf v}_i) \in \ML(\MA^{(r)})$. If ${\bf z}=[{\bf v}_1\  -{\bf v}_1\ {\bf 0}\ \ldots\ {\bf 0}]^T$ and
${\bf w}=[{\bf 0}\  {\bf v}_2+{\bf v}_1\ {\bf v}_3\ \ldots\ {\bf v}_r]^T$, it is clear that ${\bf v}={\bf z}+{\bf w}$ 
and  that $\bf z$ and $\bf w$ are also
in $ \ML(\MA^{(r)})$. The following lemma comments on the semiconformality of the 
decomposition of ${\bf v}$ into these two vectors.  
   
\begin{Lemma}\label{sc=ssc} Let 
${\bf v}\in \ML(\MA^{(r)})$, $\bf z$, $\bf w$ be as above. If $\bf v$ has type greater than or equal to $3$ and
${\bf v}={\bf z}+_{sc}{\bf w}$ is proper then ${\bf v}={\bf z}+_{ssc}{\bf w}$. Similarly if
${\bf v}={\bf w}+_{sc}{\bf z}$ is proper then ${\bf v}={\bf w}+_{ssc}{\bf z}$.
\end{Lemma}

\begin{proof} Suppose that 
 ${\bf v}={\bf z}+_{sc}{\bf w}$ is proper semiconformal. Assume by contradiction that ${\bf v}$ is not the strongly semiconformal sum of ${\bf z}$ and ${\bf w}$. It follows that ${\bf v}^+={\bf z}^+$ or ${\bf v}^-={\bf w}^-$. If ${\bf v}^+={\bf z}^+$ then we obtain that ${\bf v}_i^+={\bf 0}$ for all $i\geq 3$ and thus since ${\bf v}_i\in\ML(\MA)$ we have ${\bf v}_i={\bf 0}$ for all $i\geq 3$. Therefore  the type of $\bf v$ is less than three, a contradiction. Hence ${\bf v}^-={\bf w}^-$ and consequently ${\bf v}_1^-={\bf 0}$. This implies that ${\bf v}_1={\bf 0}$ and thus ${\bf z}={\bf 0}$, a contradiction. The assertion about the second decomposition is proved similarly. 
\end{proof}

\medskip

\section{Markov complexity for monomial curves in $\AA^3$ which are not complete intersections }

In this section we study Lawrence liftings of  monomial curves in $\AA^3$ whose corresponding toric ideals are not complete intersections. Suppose  that $\MA=\{ n_1, n_2,$ $ n_3\}$ $\subset\ZZ_{>0}$ is a monomial curve such that $ I_{\MA}$ is not a complete intersection ideal. For $1\le i\le 3$ we let $c_i$ be the smallest element of $\ZZ_{>0}$ such that $c_i n_i=r_{ij}n_j+ r_{ik}n_k$,   $r_{ij},r_{ik}\in\ZZ_{>0}$ with $\{i,j,k\}=\{1,2,3\}$.  Below, we recall the description of the unique Markov basis of $\MA$ given in \cite[Proposition 3.2, Proposition 3.3]{He}. 

\begin{Theorem}\label{notcompint}
Let $\MA=\{n_1,n_2,n_3\}$ be a set of positive integers with $\gcd(n_1,n_2,n_3)=1$, and with the property that $I_{\MA}$ is not a complete intersection ideal. Let ${\bf u}_1=(-c_1, r_{12}, r_{13})$, ${\bf u}_2=(r_{21}, -c_2, r_{23})$,  ${\bf u}_3=(r_{31}, r_{32}, -c_3)$. Then $\MA$ has a unique minimal Markov basis, 
${\mathcal M}(\MA)= \{ {\bf u}_1,{\bf u}_2, {\bf u}_3 \}$ and ${\bf u}_1+{\bf u}_2+{\bf u}_3={\bf 0}$.
\end{Theorem}

It follows immediately from Theorem~\ref{notcompint} that $\MM(\MA)=\MS(\MA)$. Another way to see that is applying \cite[Remark 4.4.3]{PS}, since $I_{\MA}$ is a generic lattice ideal.

To any vector ${\bf u}\in\ZZ^3$ we assign a sign-pattern: we put   $+$ if the coordinate of ${\bf u}$ is positive, $-$ if the coordinate of ${\bf u}$ is  negative and  $0$ if the corresponding coordinate  is $0$. When we don't know exactly the sign of the coordinate   we use the symbol $*$.
For example the vector $(0, 2, -3)$ has the sign-pattern $0+-$ while the elements ${\bf u}_1,{\bf u}_2,{\bf u}_3$ of ${\mathcal M}(\MA)$  have   sign-patterns $-++$, $+-+$ and $++-$. We note that when  $\alpha,\beta\in\ZZ_{>0}$ the vectors $\alpha(-{\bf u}_1)$, $\beta{\bf u}_2$ and $\alpha(-{\bf u}_1)+\beta{\bf u}_2$ have sign-patterns $+--,+-+$ and $+-*$. By looking at the sign patterns, it is immediate that   $\alpha(-{\bf u}_1)+\beta{\bf u}_2=\alpha(-{\bf u}_1)+_{ssc}\beta{\bf u}_2$. We generalize and isolate this remark.

\begin{Remark}
\label{semiconf_sum}
{\em Let  $i\neq j\in\{1,2,3\}$ and   $\alpha,\beta\in\ZZ_{>0}$. Then    $\alpha (-{\bf u}_i)+\beta {\bf u}_j=\alpha (-{\bf u}_i)+_{ssc}\beta {\bf u}_j$.}
\end{Remark}

\noindent The geometry of the plane implies the following lemma. 
\begin{Lemma}\label{uniquesemisum}
Let $\MA=\{n_1,n_2,n_3\}$ be such that $I_{\MA}$ is not a complete intersection. Let ${\bf 0}\neq {\bf v}\in\ML(\MA)$.
Then  either ${\bf v}=\alpha (\pm {\bf u}_i)$ for an $\alpha\in\ZZ_{>0}$ and $i\in\{1,2,3\}$ or ${\bf v}=\alpha (-{\bf u}_i) +_{ssc} \beta {\bf u}_j$ for some $\alpha,\beta\in\ZZ_{>0}$ and $i\neq j \in\{1,2,3\}$.  
\end{Lemma}

\begin{proof} Since ${\bf u}_1+{\bf u}_2+{\bf u}_3={\bf 0}$ and $\ML(\MA)=\ZZ{\bf u}_1+\ZZ{\bf u}_2+\ZZ{\bf u}_3$ it follows that $\rank_{\ZZ}(\ML(\MA))=2$. Therefore the three vectors ${\bf u}_1,{\bf u}_2,{\bf u}_3$ define a complete pointed polyhedral fan of the plane, $\span_{\RR}({\bf u}_1,{\bf u}_2)$. This polyhedral fan has thirteen non-empty faces. Other than  the vertex $\{\bf 0\}$,  six of the faces are   one-dimensional: $\RR_{+}{\bf u}_1$, $\RR_{+}(-{\bf u}_3)$, $\RR_{+}{\bf u}_2$, $\RR_{+}(-{\bf u}_1)$, $\RR_{+}{\bf u}_3$ and $\RR_{+}(-{\bf u}_2)$. The rest  are two-dimensional: $\RR_{+}\{-{\bf u}_3,{\bf u}_1\}$, $\RR_{+}\{-{\bf u}_3,{\bf u}_2\}$, $\RR_{+}\{-{\bf u}_1,{\bf u}_2\}$, $\RR_{+}\{-{\bf u}_1,{\bf u}_3\}$, $\RR_{+}\{-{\bf u}_2,{\bf u}_3\}$ and %\newline
$\RR_{+}\{-{\bf u}_2,$ ${\bf u}_1\}$.  

Let ${\bf v}\in\ML(\MA)$ be a nonzero vector. Since ${\bf v}$ belongs to the polyhedral fan then it follows at once that ${\bf v}$ belongs either to one of the six one-dimensional cones, that is ${\bf v}=\alpha (\pm {\bf u}_i)$ for some $\alpha\in\ZZ_{>0}$ and $i\in\{1,2,3\}$, or ${\bf v}$ belongs to the interior of one of the six two-dimensional cones, that is $\alpha (-{\bf u}_i)+\beta {\bf u}_j$ for some $\alpha,\beta\in \ZZ_{>0}$ and $i\neq j\in\{1,2,3\}$. It is immediate, by Remark~\ref{semiconf_sum}, that this sum determines a proper strongly semiconformal decomposition of ${\bf v}$. 
\end{proof}

 In the following lemma we consider a subset  $T$  of $\ML(\MA^{(r)})$. The elements of $T$ have type $2$ and $3$.
In Theorem~\ref{markovnotci} we will see that $T$ equals  the universal Markov basis of $\MA^{(r)}$.

\begin{Lemma}\label{indispensablemon}
Let $\MA=\{n_1,n_2,n_3\}$ be such that $I_{\MA}$ is not a complete intersection. Let $r\geq 3$  and let $T$ be the subset of $\ML(\MA^{(r)})$ containing all 
vectors of type $2$ whose nonzero rows are of the form ${\bf u},-{\bf u}$, with ${\bf u}\in \MG(\MA)$ and all vectors of type $3$ whose  nonzero rows  are permutations of ${\bf u}_1,{\bf u}_2,{\bf u}_3$. Then $T\subset\MS(\MA^{(r)})$. Moreover  $|T|=k\binom{r}{2}+6\binom{r}{3}$, where $k$ is the cardinality of the Graver basis of $\MA$.  
\end{Lemma}
\begin{proof}
We will show  that $T\subset \MS(\MA^{(r)})$, the last part of the conclusion being immediate. By Remark~\ref{declaw} we may assume that $r=3$.  We first show that the element of type 2  
\[
\begin{bmatrix}  {\bf u}\cr -{\bf u}\cr {\bf 0}\end{bmatrix}, {\bf u}\in \MG(\MA)
\]
is indispensable. Suppose that is not. Then by Proposition~\ref{indispensabledec} it admits a proper semiconformal decomposition, which is of the following form due to  Remark~\ref{declaw}  
\[ 
\begin{bmatrix}  {\bf u}\cr -{\bf u}\cr {\bf 0}\end{bmatrix}= \begin{bmatrix}  {\bf v}_1\cr -{\bf v}_1\cr {\bf 0}\end{bmatrix}+_{sc} \begin{bmatrix}  {\bf v}_2\cr -{\bf v}_2\cr {\bf 0}\end{bmatrix}\ . 
\]
Thus ${\bf u}={\bf v}_1+_{sc} {\bf v}_2$ and $-{\bf u}=-{\bf v}_1+_{sc} (-{\bf v}_2)$. But ${\bf u}={\bf v}_1+_{sc} {\bf v}_2$  also implies that $-{\bf u}=-{\bf v}_2+_{sc} {\bf v}_1$ and thus ${\bf u}={\bf v}_1+_{c} {\bf v}_2$. Thus ${\bf u}$ is not in the Graver basis of $\MA$, a contradiction. Therefore all elements of type 2 whose nonzero rows belong to the Graver basis of $\MA$ are indispensable. Next we prove that the element of type 3 
\[ 
\begin{bmatrix}  {\bf u}_1\cr {\bf u}_2\cr {\bf u}_3\end{bmatrix}
\]
is indispensable. Arguing by contradiction and applying again Proposition~\ref{indispensabledec}, it would imply that the vector has the following  proper semiconformal decomposition 
\[
 \begin{bmatrix}  {\bf u}_1\cr {\bf u}_2\cr {\bf u}_3\end{bmatrix}=\begin{bmatrix}  {\bf v}_1\cr {\bf v}_2\cr {\bf v}_3\end{bmatrix}+_{sc}\begin{bmatrix}  {\bf w}_1\cr {\bf w}_2\cr {\bf w}_3\end{bmatrix}. 
\]
Hence we have ${\bf u}_i={\bf v}_i+_{sc}{\bf w}_i$ for $i=1,2,3$. Since ${\bf u}_i\in\MS(\MA)$ it follows that either ${\bf v}_i$ or ${\bf w}_i$ is zero. It follows that one of the vectors $[ {\bf v}_1\ {\bf v}_2\  {\bf v}_3]^T$ or $[ {\bf w}_1\ {\bf w}_2\  {\bf w}_3]^T$ has at least two zero rows, and thus it must be zero. This is a contradiction and consequently all elements of type $3$ with the rows being permutations of ${\bf u}_1, {\bf u}_2, {\bf u}_3$ are indispensable.  
\end{proof} 

In the next lemma we show that a class of elements of $\ML(\MA^{(r)})$ is not part of  $\MM(\MA^{(r)})$. 
 
\begin{Lemma}\label{kindtwo}
Let $\MA=\{n_1,n_2,n_3\}$ be such that $I_{\MA}$ is not a complete intersection. Let $r\geq 3$ and $T\subset\MS(\MA^{(r)})$ be as  in Lemma~\ref{indispensablemon}. Suppose that for $1\leq i\leq r$,  $\alpha_i,\beta_i\in\NN$, 
\[ {\bf 0}\neq {\bf v}=
\begin{bmatrix}  \alpha_1 (-{\bf u}_1) +\beta_1 {\bf u}_2\cr \vdots \cr \alpha_s (-{\bf u}_1) +\beta_s {\bf u}_2\cr \beta_{s+1} (-{\bf u}_2) +\alpha_{s+1} {\bf u}_1\cr \vdots \cr \beta_{r} (-{\bf u}_2) +\alpha_{r} {\bf u}_1 \end{bmatrix} \in \ML(\MA^{(r)})\setminus T\ .
\]
Then $\bf v$ has a proper strongly semiconformal decomposition into summands of $\ML(\MA^{(r)})$ where   one of them has type $2$.   
\end{Lemma}
\begin{proof} First suppose that  ${\bf v}$ is of a type $2$. Since ${\bf v}\not\in T$   it follows that its nonzero rows do not belong to $\MG(\MA)$, and thus ${\bf v}$ admits a proper conformal decomposition. Applying Lemma~\ref{c_sc_ssc}(i) it follows that ${\bf v}$ has also a proper strongly semiconformal decomposition. Next suppose that the type of ${\bf v}$ is greater than or equal to $3$.  By Remark~\ref{declaw} we can assume that the type of ${\bf v}$ is exactly $r$ where $r\geq 3$. We denote the   row vectors of ${\bf v}$ by ${\bf v}_1,\ldots,{\bf v}_r$. We  notice  that $1\leq s<r$. Indeed, if $s=0$ or $s=r$ then we obtain that $(\sum_{i=1}^r \alpha_i){\bf u}_1 - (\sum_{i=1}^r \beta_i){\bf u}_2={\bf 0}$. Since ${\bf u}_1,{\bf u}_2$ are linearly independent this implies that $\sum_{i=1}^r \alpha_i=\sum_{i=1}^r \beta_i=0$. Therefore we have $\alpha_i=\beta_i=0$ for all $i$, which leads to ${\bf v}={\bf 0}$, a contradiction.

Suppose that $s=1$. Since the row vectors of ${\bf v}$ add up to zero and ${\bf u}_1,{\bf u}_2$ are linearly independent, we obtain that $\alpha_1=\sum_{i\geq 2}\alpha_i$ and $\beta_1=\sum_{i\geq 2}\beta_i$. Note that $\alpha_1\neq 0$ or $\beta_1\neq 0$, otherwise we would have ${\bf v}={\bf 0}$, a contradiction. If $\alpha_1=0$ then $\alpha_i=0$ for all $i$.  Let ${\bf z}=[{\bf u}_2\ -{\bf u}_2\ {\bf 0}\ \ldots\ {\bf 0}]^T$. Then ${\bf v}={\bf z}+_{ssc} ({\bf v}-{\bf z})$. The case $\beta_1=0$ is similar. Furthermore, if $\alpha_1,\beta_1\neq {\bf 0}$ we notice that ${\bf v}$ can be decomposed as the sum of two vectors ${\bf z},{\bf w}$ of type $2$ and $r-1$: 
\[
{\bf z}=
\begin{bmatrix}  \alpha_2 (-{\bf u}_1) +\beta_2 {\bf u}_2\cr \beta_{2} (-{\bf u}_2) +\alpha_{2} {\bf u}_1\cr {\bf 0} \cr \vdots \cr {\bf 0} \end{bmatrix}, \  {\bf w}= \begin{bmatrix}  (\sum_{i\geq 3} \alpha_i) (-{\bf u}_1) +(\sum_{i\geq 3}\beta_i) {\bf u}_2\cr {\bf 0} \cr \beta_3 (-{\bf u}_2) + \alpha_3 {\bf u}_1 \cr \vdots \cr \beta_{r} (-{\bf u}_2) +\alpha_{r} {\bf u}_1 \end{bmatrix}.
\]      
The sign-patterns of $\alpha_2 (-{\bf u}_1) +\beta_2 {\bf u}_2$ and $(\sum_{i\geq 3} \alpha_i) (-{\bf u}_1) +(\sum_{i\geq 3}\beta_i) {\bf u}_2$ are of the form $+-*$. Thus ${\bf v}={\bf z}+_{sc}{\bf w}$ or ${\bf v}={\bf w}+_{sc}{\bf z}$. Applying Lemma~\ref{sc=ssc} it follows that ${\bf v}={\bf z}+_{ssc}{\bf w}$ or ${\bf v}={\bf w}+_{ssc}{\bf z}$. 

We remark that if $s=r-1$, then a similar argument holds. Suppose now that    $2\leq s\leq r-2$. Since ${\bf v}\in\ML(\MA^{(r)})$ and ${\bf u}_1,{\bf u}_2$ are linearly independent we have the following relations
\[
\sum_{i=1}^s \alpha_i=\sum_{j=s+1}^r \alpha_j \ \text{ and } \ \sum_{i=1}^s \beta_i=\sum_{j=s+1}^r \beta_j.
\] 
We may assume for the rest of the proof that $s\leq r-s$, otherwise we replace ${\bf v}$ by $-{\bf v}$. Since $s\leq r-s$ there exist $i,j$ with $1\leq i\leq s$ and $s+1\leq j\leq r$ such that $\alpha_i\geq \alpha_j$. For  simplicity of notation we may assume that $i=1$ and $j=r$. First suppose that $\alpha_1=\alpha_r$. If $\beta_1\leq \beta_r$  then 
\[
{\bf v}= \begin{bmatrix}  {\bf 0}\cr \alpha_{2} (-{\bf u}_1) + \beta_{2} {\bf u}_2\cr \vdots\cr \beta_{r-1} (-{\bf u}_2) + \alpha_{r-1} {\bf u}_1 \cr (\beta_r-\beta_1)(-{\bf u}_2)  \end{bmatrix} +_{sc} \begin{bmatrix}  \alpha_1 (-{\bf u}_1) +\beta_1 {\bf u}_2\cr {\bf 0}\cr \vdots\cr {\bf 0} \cr  \beta_{1} (-{\bf u}_2) +\alpha_{r} {\bf u}_1 \end{bmatrix},
\]
otherwise 
\[
{\bf v}= \begin{bmatrix}  \alpha_1 (-{\bf u}_1) +\beta_r {\bf u}_2\cr {\bf 0}\cr \vdots\cr {\bf 0} \cr  \beta_{r} (-{\bf u}_2) +\alpha_{r} {\bf u}_1 \end{bmatrix} +_{sc} \begin{bmatrix} (\beta_1-\beta_r){\bf u}_2 \cr \alpha_{2} (-{\bf u}_1) + \beta_{2} {\bf u}_2\cr \vdots\cr \beta_{r-1} (-{\bf u}_2) + \alpha_{r-1} {\bf u}_1 \cr {\bf 0}  \end{bmatrix}.
\]
In both situations we may apply Lemma~\ref{sc=ssc} and we obtain that ${\bf v}$ has a proper strongly semiconformal decomposition. Next  we examine what happens when $\alpha_1>\alpha_r$.  If $\beta_1\geq\beta_r$ then ${\bf v}$ can be written as a sum of two vectors ${\bf z},{\bf w}$ of types $2$ and $r-1$:
\[
{\bf z}= \begin{bmatrix}  \alpha_r (-{\bf u}_1) +\beta_r {\bf u}_2\cr {\bf 0}\cr \vdots\cr {\bf 0} \cr  \beta_{r} (-{\bf u}_2) +\alpha_{r} {\bf u}_1 \end{bmatrix}, \ {\bf w}=\begin{bmatrix}  (\alpha_1-\alpha_r) (-{\bf u}_1) +(\beta_1-\beta_r) {\bf u}_2\cr \alpha_{2} (-{\bf u}_1) + \beta_{2} {\bf u}_2\cr \vdots\cr \beta_{r-1} (-{\bf u}_2) + \alpha_{r-1} {\bf u}_1 \cr {\bf 0}  \end{bmatrix}.
\]
Since the type of ${\bf v}$ is $r$, the type of ${\bf z}$ is $2$. Moreover since $\alpha_1-\alpha_r>0$ and $\beta_1-\beta_r\geq 0$,   the sign-pattern of $(\alpha_1-\alpha_r) (-{\bf u}_1) +(\beta_1-\beta_r) {\bf u}_2$ is $+-*$, the same as the sign-pattern of $\alpha_r (-{\bf u}_1) +\beta_r {\bf u}_2$. Therefore the two vectors   add up semiconformally either in this order or the the reversed. Hence ${\bf v}={\bf z}+_{sc} {\bf w}$ or ${\bf v}={\bf w}+_{sc} {\bf z}$ and applying Lemma~\ref{sc=ssc} we see that  the  sums are also strongly semiconformal. The last case to  consider is when  $\beta_1<\beta_r$. If $\alpha_r+\beta_1>0$ then ${\bf v}$ can be decomposed as the semiconformal sum of  ${\bf w},{\bf z}$ where  
\[{\bf w}=
 \begin{bmatrix}  (\alpha_1-\alpha_r) (-{\bf u}_1) \cr \alpha_{2} (-{\bf u}_1) + \beta_{2} {\bf u}_2\cr \vdots\cr \beta_{r-1} (-{\bf u}_2) + \alpha_{r-1} {\bf u}_1 \cr (\beta_r-\beta_1)(-{\bf u}_2) \end{bmatrix} , \ {\bf z}=\begin{bmatrix}  \alpha_r (-{\bf u}_1) +\beta_1 {\bf u}_2\cr {\bf 0}\cr \vdots\cr {\bf 0} \cr  \beta_1 (-{\bf u}_2) +\alpha_{r} {\bf u}_1 \end{bmatrix}.  
\]  
Indeed, the the sign-patterns of the first and last row vector of each summand force the sum to be semiconformal. We show that  the sum is also strongly semiconformal. Indeed, if ${\bf v}^-={\bf z}^-$ then ${\bf v}_k^-={\bf 0}$  and thus ${\bf v}_k={\bf 0}$ for all $k$ with $2\leq k\leq r-1$, a contradiction since the type of ${\bf v}$ is bigger than $2$. On the other hand, if ${\bf v}^+={\bf w}^+$ then ${\bf v}_1^+={\bf w}_1^+$ and the contradiction follows from the sign-patterns of ${\bf w}_1$ and ${\bf z}_1$ which are $+-*$. So, ${\bf v}$ can be written as a strongly semiconformal sum.

Finally suppose that    $\alpha_r+\beta_1=0$, that is $\alpha_r=\beta_1=0$. Let $j_0\in\{s+1,\ldots,r\}$ be such that $\alpha_{j_0}>0$.  We may assume that $j_0=r-1$. Then
\[
{\bf v}= 
\begin{bmatrix}  (\alpha_1-1) (-{\bf u}_1) \cr \alpha_{2} (-{\bf u}_1) + \beta_{2} {\bf u}_2\cr \vdots\cr \beta_{r-1} (-{\bf u}_2) + (\alpha_{r-1}-1) {\bf u}_1 \cr  \beta_{r} (-{\bf u}_2) \end{bmatrix} +_{ssc} \begin{bmatrix}  -{\bf u}_1 \cr {\bf 0} \cr \vdots\cr {\bf u}_1 \cr {\bf 0} \end{bmatrix}.
\] 
\end{proof}

It follows from Lemma~\ref{uniquesemisum} that for an arbitrary ${\bf v}\in\ML(\MA^{(r)})$, whose row vectors are ${\bf v}_1,\ldots,{\bf v}_r$, we can associate in a unique way a vector  $(\alpha,\alpha',\beta,\beta',\gamma,\gamma')\in\NN^6$, which we denote  by $\lambda({\bf v})$, such that $\alpha$ is the coefficient of ${\bf u}_1$ in $\sum_{i=1}{\bf v}_i$, $\alpha'$ the coefficient of $-{\bf u}_1$, $\beta$ the coefficient of ${\bf u}_2$, $\beta'$ the coefficient of $-{\bf u}_2$ and so on. On the other hand since $\sum_{i=1}^r {\bf v}_i={\bf 0}$ we obtain that
\[
(\alpha-\alpha'){\bf u}_1+(\beta-\beta'){\bf u}_2+(\gamma-\gamma'){\bf u}_3=0.
\]
By  replacing in the equation above ${\bf u}_3=-{\bf u}_1-{\bf u}_2$ and by the linear independence of 
${\bf u}_1,{\bf u}_2$ we see that  $\alpha-\alpha'=\beta-\beta'=\gamma-\gamma'$.
\begin{Theorem}\label{markovnotci}
Let $\MA=\{n_1,n_2,n_3\}$ be such that $I_{\MA}$ is not a complete intersection. Then $m(\MA)$, the Markov complexity of $\MA$, is 3.  Moreover, for any $r\geq 3$ we have $\MM(\MA^{(r)})=\MS(\MA^{(r)})$ and the cardinality of $\MM(\MA^{(r)})$ is $k\binom{r}{2}+6\binom{r}{3}$, where $k$ is the cardinality of the Graver basis of $\MA$. 
\end{Theorem}

\begin{proof} We let $T$ be the subset of elements of $\MS(\MA^{(r)})$   described in Lemma~\ref{indispensablemon}. We will show that every nonzero vector ${\bf v}\in\ML(\MA^{(r)})\setminus T$ admits a proper strongly semiconformal decomposition. This will imply via Proposition~\ref{universalmarkov} that $\MM(\MA^{(r)})\subset T$ and thus $\MS(\MA^{(r)})=\MM(\MA^{(r)})=T$. Let ${\bf v}\in\ML(\MA^{(r)})\setminus T$ be a nonzero vector such that $\lambda({\bf v})=(\alpha,\alpha',\beta,\beta',\gamma,\gamma')$. We have two main cases: $\alpha\beta\gamma\neq 0$ and $\alpha\beta\gamma=0$. 

Assume first that $\alpha\beta\gamma\neq 0$. It follows from Lemma~\ref{uniquesemisum} that there exist three different rows of ${\bf v}$, and we may assume that are the first three, such that in their unique semiconformal decomposition we have a positive multiple of ${\bf u}_1,{\bf u}_2,{\bf u}_3$. Without loss of generality, we may assume that  for  $1\le i\le 3$, ${\bf u}_i$ appears in   row $i$, that is ${\bf v}_i=a_i(-{\bf u}_{j_i})+_{sc}b_i{\bf u}_i$ where $j_i\neq i$ and $a_i,b_i\in\NN$ with $b_i\neq 0$. We prove now that ${\bf v}={\bf z}+_{ssc}{\bf w}$, where ${\bf w}=({\bf u}_1\ {\bf u}_2\ {\bf u}_3\ {\bf 0}\ \ldots\ {\bf 0})^T$ and ${\bf z}={\bf v}-{\bf w}$. 
Note first that ${\bf z}$ is nonzero since ${\bf v}\not\in T$. For $1\le i\le 3$, the $i^{th}$ row of $\bf z$ is equal to $a_i(-{\bf u}_{j_i})+(b_i-1){\bf u}_i$,  while all other rows of $\bf z$ are equal to the corresponding rows of $\bf v$. By the sign-patterns it is easy to see that the first three rows of $\bf v$ have a semiconformal decomposition:
\[ a_i(-{\bf u}_{j_i})+b_i{\bf u}_i=(a_i(-{\bf u}_{j_i})+(b_i-1){\bf u}_i)+_{sc} {\bf u}_i,\ 1\le i\le 3\ .\]
When $i=1$,   the sign-pattern of ${\bf u}_1$ is $-++$ and thus $(a_1(-{\bf u}_{j_1})+b_1{\bf u}_1)^+>(a_1(-{\bf u}_{j_1})+(b_1-1){\bf u}_1)^+$. Therefore  ${\bf v}^+>{\bf z}^+$ and ${\bf v}={\bf z}+_{sc}{\bf w}$.  Finally, assume by contradiction that the decomposition is not strongly semiconformal. 
Then we necessarily have ${\bf v}^-={\bf w}^-$.  Since ${\bf w}_1(1)={\bf u}_1(1)<0$ this implies ${\bf z}_1(1)=0$. On the other hand ${\bf z}_1=a_1(-{\bf u}_{j_1})+_{sc}(b_1-1){\bf u}_1$ and since $j_1\in\{2,3\}$ then the sign-pattern of ${\bf z}_1$ is $-**$ if $a_1\neq 0$ or $b_1\neq 1$. Thus ${\bf z}_1(1)=0$ implies $a_1=0$ and $b_1=1$ and so ${\bf z}_1={\bf 0}$. Similarly we have that ${\bf z}_2={\bf z}_3={\bf 0}$. Since ${\bf w}_k={\bf 0}$ for all $k\geq 4$ and ${\bf v}^-={\bf w}^-$ it follows that ${\bf v}_k^-={\bf 0}$ for all $k\geq 4$. Therefore ${\bf v}_k={\bf v}_k^+$ for all $k\geq 4$ and using the fact that ${\bf v}_k\in\ML(\MA)$ we obtain ${\bf v}_k=0$ for all $k\geq 4$. This implies ${\bf z}_k={\bf 0}$ for all $k\geq 4$ and consequently ${\bf v}={\bf w}\in T$, a contradiction. Thus we have proved that ${\bf v}={\bf z}+_{ssc}{\bf w}$, as desired. We should note now that if $\alpha'\beta'\gamma'\neq 0$ then a similar argument shows that ${\bf v}={\bf w}'+_{ssc}{\bf z}'$, where ${\bf w}'=-{\bf w}$.       

In the second case we have $\alpha\beta\gamma=0$ and from the above remark we may also assume that $\alpha'\beta'\gamma'=0$. Without loss of generality we let $\alpha=0$. If $\alpha'=0$ then it follows from Lemma~\ref{uniquesemisum} that ${\bf v}$ is of the form described in Lemma~\ref{kindtwo} and we are done. Otherwise $\alpha'\neq 0$ and since $\alpha'\beta'\gamma'=0$ we have $\beta'=0$ or $\gamma'=0$. We analyze just the case $\beta'=0$, the other one being analogous. It follows from the definition of $\lambda({\bf v})$ that $\alpha'(-{\bf u}_1)+\beta {\bf u}_2 +\gamma {\bf u}_3 +\gamma'(-{\bf u}_3)={\bf 0}$ and using the relation $-{\bf u}_1={\bf u}_2+{\bf u}_3$ we get from the linear independence of ${\bf u}_2$ and ${\bf u}_3$ that $\alpha'+\beta=0$. Thus, since $\alpha',\beta\geq 0$ we have $\alpha'=\beta=0$, a contradiction. Therefore we obtain that $T=\MM(\MA^{(r)})$ and thus $m(\MA)=3$.   
\end{proof}
 
\begin{Remark}\label{bound_rem_not_com}{\rm In \cite{HS},  a  lower bound for $m(\MA)$ is given  by Ho\c sten and Sullivant for $\MA\subset\NN^m$: it equals the maximum $1$-norm of the elements in the Graver basis of $\MS(\MA)$, see \cite[Theorem 3.11]{HS}. By looking at the explicit description of the elements of $\MS(\MA)$ as given in Theorem~\ref{notcompint}, it is easy to see that the Graver basis of  $\MS(\MA)$ contains exactly one element: $(1,1,1)$. The $1$-norm of this element is $3$ and thus in this case the lower bound of \cite[Theorem 3.11]{HS} actually equals $m(\MA)$. 
}
\end{Remark}

\section{Markov complexity for monomial curves in $\AA^3$ which are complete intersections} 

In this section we study Lawrence liftings of  monomial curves in $\AA^3$ whose corresponding toric ideals are complete intersections. Suppose  that $\MA=\{ n_1, n_2,$ $ n_3\}$ $\subset\ZZ_{>0}$ is a monomial curve such that $I_{\MA}$ is  a complete intersection ideal. For $1\le i\le 3$ we let $c_i$ be the smallest element of $ \ZZ_{>0}$ such that $c_i n_i=r_{ij}n_j+ r_{ik}n_k$,   $r_{ij},r_{ik}\in  \NN$ with $\{i,j,k\}=\{1,2,3\}$.  In \cite[Proposition 3.4]{He},  it was shown that either $(0,-c_2,c_3)\in\MM(\MA)$ or $(c_1,0,-c_3)\in\MM(\MA)$ or $(-c_1,c_2,0)\in\MM(\MA)$. Below, we recall the description of the universal Markov basis of $\MA$ given  in \cite[Proposition 3.5]{He} when $(0,-c_2,c_3)\in\MM(\MA)$.

\begin{Proposition}\label{completeint}
 Let $\MA=\{n_1,n_2,n_3\}$ be a set of positive integers such that $\gcd(n_1,n_2,n_3)=1$,  $I_{\MA}$ is a complete intersection and  $(0,-c_2,c_3)\in\MM(\MA)$.  Let ${\bf u}_1=(-c_1,r_{12},r_{13})$ and ${\bf u}_2=(0,-c_2,c_3)$.  The universal Markov basis of $\MA$ is 
\[
\mathcal M(\MA)=\{{\bf u}_2,d\cdot {\bf u}_2+{\bf u}_1 :\  \ -\lfloor\frac{r_{13}}{c_3}\rfloor \leq d\leq\lfloor\frac{r_{12}}{c_2}\rfloor\},
\] 
\end{Proposition}

For the rest of this section we will  assume that $(0,-c_2,c_3)\in\MM(\MA)$,  the other two cases being similar. Let us note   that the sign patterns of   ${\bf u}_1$ and ${\bf u}_2$ are $-++$ and $0-+$. 

\begin{Lemma}\label{sccompint}
Let $\MA=\{n_1,n_2,n_3\}$ be such that $I_{\MA}$ is a complete intersection and
$(0,-c_2,c_3)\in\MM(\MA)$. If ${\bf 0}\neq {\bf v}\in\ML(\MA)$ then there are unique $\alpha,\beta\in\NN$ such that
  ${\bf v}=\alpha (-{\bf u}_1) +_{sc} \beta (\pm {\bf u}_2)$ or ${\bf v}=\beta (\pm {\bf u}_2) +_{sc} \alpha{\bf u}_1$.
\end{Lemma}
\begin{proof}
Since $\rank_{\ZZ}(\ML({\MA}))=2$, if ${\bf v}\neq {\bf 0}$ then there are $a,b\in \ZZ$ such that 
 ${\bf v}=a{\bf u}_1+b{\bf u}_2$ for some integers $a,b$. It is easy to see how to write $\bf v$ in the desired 
form.
\end{proof}

We note that the semiconformal decomposition of Lemma~\ref{sccompint} might not be strongly semiconformal as is the case for  $(c_3+1)(-{\bf u}_1)+_{sc}{\bf u}_2$. Similarly, one can show that any of the semiconformal sums from Lemma~\ref{sccompint} with $\alpha,\beta\in\ZZ_{>0}$ may not be in general strongly semiconformal. In the next lemma we identify certain elements which are not part of the universal Markov basis of $\MA^{(r)}$. 

\begin{Lemma}\label{kindtwoci}
Let $\MA=\{n_1,n_2,n_3\}$ be such that $I_{\MA}$ is a complete intersection and  $(0,-c_2,c_3)\in\MM(\MA)$. Let  $r\geq 3$. Suppose that  ${\bf 0}\neq {\bf v}\in\ML(\MA^{(r)})$ is of the following form
\[{\bf v}=
\begin{bmatrix}  \alpha_1 (-{\bf u}_1) +\beta_1 {\bf u}_2\cr \vdots \cr \alpha_s (-{\bf u}_1) +\beta_s {\bf u}_2\cr \beta_{s+1} (-{\bf u}_2) +\alpha_{s+1} {\bf u}_1\cr \vdots \cr \beta_{r} (-{\bf u}_2) +\alpha_{r} {\bf u}_1 \end{bmatrix},
\]
where $1\leq s\leq r-1$ and $\alpha_i,\beta_i\in\ZZ_{>0}$  for $1\leq i\leq r$. Then ${\bf v}$  admits a proper strongly semiconformal decomposition into two vectors, where one of the summands is a vector of type two.   
\end{Lemma}
\begin{proof}
Note first that our hypotheses imply that the type of ${\bf v}$ is $r$. Since ${\bf v}\in\ML(\MA^{(r)})$, the sum of the rows of ${\bf v}$ is zero. Since  ${\bf u}_1,{\bf u}_2$ are linearly independent we have 
\begin{equation}\label{need}
\sum_{i=1}^s\alpha_i=\sum_{j=s+1}^r\alpha_j \ \text{ and } \ \sum_{i=1}^s \beta_i=\sum_{j=s+1}^r \beta_j.
\end{equation}
We will analyze two cases. For the first case we assume that there exist integers $i,j$ corresponding to nonzero rows of ${\bf v}$ with $1\leq i\leq s$ and $s+1\leq j\leq r$ such that either a) $\alpha_i\leq \alpha_j$ and $\beta_i\leq \beta_j$ or b) $\alpha_i\geq \alpha_j$ and $\beta_i\geq\beta_j$. Without loss of generality we may assume that $i=1$, $j=r$. Let us first suppose that   $\alpha_1\leq\alpha_r$ while $\beta_1\leq\beta_r$. Then ${\bf v}={\bf z}+{\bf w}$, where
\[
{\bf z} = \begin{bmatrix}  \alpha_1 (-{\bf u}_1) +\beta_1 {\bf u}_2\cr {\bf 0}\cr \vdots \cr {\bf 0} \cr \beta_{1} (-{\bf u}_2) +\alpha_{1} {\bf u}_1 \end{bmatrix} ,\ {\bf w}= \begin{bmatrix} {\bf 0} \cr \alpha_2 (-{\bf u}_1) +\beta_2 {\bf u}_2\cr \vdots \cr \beta_{r-1} (-{\bf u}_2) +\alpha_{r-1} {\bf u}_1 \cr (\beta_{r}-\beta_1) (-{\bf u}_2) +(\alpha_r-\alpha_{1}) {\bf u}_1 \end{bmatrix} \ .
\]
We will show that either ${\bf v}={\bf z}+_{ssc}{\bf w} $ or ${\bf v}={\bf w}+_{ssc}{\bf z} $. Indeed, note first that by the assumptions ${\bf z}$ is a type two vector and ${\bf w}$ is nonzero since ${\bf v}$ is of type $r\geq 3$. Moreover, since the sign-pattern of $\beta_1 (-{\bf u}_2) +\alpha_1 {\bf u}_1$ is $-+*$, for the $r^{th}$ row of ${\bf v}$ we have either ${\bf v}_r={\bf z}_r+_{sc}{\bf w}_r$  or ${\bf v}_r={\bf w}_r+_{sc}{\bf z}_r$ . Indeed, since $\beta_r-\beta_1\geq 0$ and $\alpha_r-\alpha_1\geq 0$   there are four possible patterns for the vector ${\bf w}_r$ corresponding to the cases: $\alpha_1=\alpha_r$ and $\beta_1=\beta_r$, $\alpha_1=\alpha_r$ and $\beta_1<\beta_r$, $\alpha_1<\alpha_r$ and $\beta_1=\beta_r$,    $\alpha_1<\alpha_r$ and $\beta_1<\beta_r$. More precisely, the corresponding four sign-patterns of ${\bf w}_r$ are: $000$, $0+-$, $-++$ and $-+*$ and one can notice that ${\bf z}_r$ and ${\bf w}_r$  add up semiconformally in this order or the reversed. This implies that ${\bf v}$ is the semiconformal sum of ${\bf z}$ and ${\bf w}$ in this order or the reversed one. Applying now Lemma~\ref{sc=ssc} we obtain that ${\bf v}$ is the strongly semiconformal sum of ${\bf z}$ and ${\bf w}$. Next let us suppose that  $\alpha_1\geq\alpha_r$ while $\beta_1\geq\beta_r$. Thus ${\bf v}={\bf z}+{\bf w}$, where
\[
{\bf z} = \begin{bmatrix}  \alpha_r (-{\bf u}_1) +\beta_r {\bf u}_2\cr {\bf 0}\cr \vdots \cr {\bf 0} \cr \beta_{r} (-{\bf u}_2) +\alpha_{r} {\bf u}_1 \end{bmatrix} ,\ {\bf w}= \begin{bmatrix} (\alpha_1-\alpha_{r})(-{\bf u}_1) + (\beta_{1}-\beta_r) {\bf u}_2 \cr \alpha_2 (-{\bf u}_1) +\beta_2 {\bf u}_2\cr \vdots \cr \beta_{r-1} (-{\bf u}_2) +\alpha_{r-1} {\bf u}_1 \cr {\bf 0} \end{bmatrix}\ .
\]
Since $\alpha_r\beta_r>0$, it follows that  the sign-pattern of ${\bf z}_1$ is $+*-$. The four possible sign-patterns of ${\bf w}_1$ are $000$, $+--$, $0+-$ and $+*-$. As before one   can conclude that either ${\bf v}={\bf z}+_{ssc}{\bf w} $ or ${\bf v}={\bf w}+_{ssc}{\bf z} $.

For the second case  the following holds for all integers $i,j$ corresponding to nonzero rows of ${\bf v}$ with $1\leq i\leq s$ and $s+1\leq j\leq r$:  either $\alpha_i<\alpha_j$ and $\beta_i>\beta_j$ or $\alpha_i>\alpha_j$ and $\beta_i<\beta_j$. Consider $i=1$ and $j=r$. Let us first suppose that   $\alpha_1<\alpha_r$ and $\beta_1>\beta_r$. We will show that  ${\bf v}={\bf z}+_{ssc}{\bf w}$, where
\[
{\bf z} = \begin{bmatrix}  \alpha_1 (-{\bf u}_1) +\beta_r {\bf u}_2\cr {\bf 0}\cr \vdots \cr {\bf 0} \cr \beta_{r} (-{\bf u}_2) +\alpha_{1} {\bf u}_1 \end{bmatrix},\ {\bf w}= \begin{bmatrix} (\beta_{1}-\beta_r) {\bf u}_2 \cr \alpha_2 (-{\bf u}_1) +\beta_2 {\bf u}_2\cr \vdots \cr \beta_{r-1} (-{\bf u}_2) +\alpha_{r-1} {\bf u}_1 \cr (\alpha_r-\alpha_1){\bf u}_1 \end{bmatrix}.
\]
Indeed, we notice first that ${\bf v}={\bf z}+_{sc}{\bf w}$, since the sign-patterns of ${\bf z}_1$ and ${\bf z}_r$ are  $+-*$ and $-+*$, while the sign patterns of ${\bf w}_1$ and ${\bf w}_r$ are $0-+$ and $-++$. Moreover the sign patterns of ${\bf z}_r$ and ${\bf w}_r$ show also that ${\bf v}_r^-\neq{\bf w}_r^-$ and implicitly ${\bf v}^-\neq{\bf w}^-$. In addition, ${\bf z}_2=\cdots={\bf z}_{r-1}={\bf 0}$ imply that ${\bf v}^+\neq{\bf z}^+$, otherwise we would have ${\bf v}_2=\cdots={\bf v}_{r-1}={\bf 0}$ and thus the type of ${\bf v}$ is two, a contradiction. Therefore we also have ${\bf v}^+\neq{\bf z}^+$ and ${\bf v}^-\neq{\bf w}^-$ and thus ${\bf v}={\bf z}+_{ssc}{\bf w}$, as desired. Next we consider the case where $\alpha_1>\alpha_r$ and $\beta_1<\beta_r$. Then ${\bf v}={\bf w}+_{sc}{\bf z}$, where
\[
{\bf w} = \begin{bmatrix} (\alpha_1-\alpha_r)(-{\bf u}_1) \cr \alpha_2 (-{\bf u}_1) +\beta_2 {\bf u}_2\cr \vdots \cr \beta_{r-1} (-{\bf u}_2) +\alpha_{r-1} {\bf u}_1 \cr (\beta_{r}-\beta_1)(-{\bf u}_2) \end{bmatrix} ,\ {\bf z}= \begin{bmatrix}  \alpha_r (-{\bf u}_1) +\beta_1 {\bf u}_2\cr {\bf 0}\cr \vdots \cr {\bf 0} \cr \beta_{1} (-{\bf u}_2) +\alpha_{r} {\bf u}_1 \end{bmatrix},
\] 
since the the sign-patterns of ${\bf w}_1$ and ${\bf w}_r$ are  $+--$ and $0+-$, while the sign patterns of ${\bf z}_1$ and ${\bf z}_r$ are $+-*$ and $-+*$. An  argument similar to the previous one shows that the sum is also strongly semiconformal, that is ${\bf v}={\bf w}+_{ssc}{\bf z}$.    
\end{proof}

We are ready to prove the main result of this section.  

\begin{Theorem}\label{markcompci}
Let $\MA=\{n_1,n_2,n_3\}$ be such that $I_{\MA}$ is a complete intersection. Then  $m(\MA)$,  the Markov complexity of
$\MA$, is  $2$. Moreover, for any $r\geq 2$ we have $\MM(\MA^{(r)})=\MS(\MA^{(r)})$ and the cardinality of $\MM(\MA^{(r)})$ is $k\binom{r}{2}$, where $k$ is the cardinality of the Graver basis of $\MA$.
\end{Theorem}
\begin{proof}

We denote by $T$ the set of type two vectors from $\ML(\MA^{(r)})$ whose nonzero rows are ${\bf u},-{\bf u}$, where ${\bf u}\in \MG(\MA)$. If $r=2$ then $\MS(\MA^{(2)})=\MM(\MA^{(2)})=\MG(\MA^{(2)})=T$, by \cite[Theorem 7.1]{St}, hence the conclusion follows immediately. Hence we may assume that $r\geq 3$. In general we have $T\subseteq\MS(\MA^{(r)})\subseteq\MM(\MA^{(r)})$ for any $r\geq 2$, so it remains to prove only that $\MM(\MA^{(r)})\subseteq T$ to get the desired conclusion. The latter will follow via Proposition~\ref{universalmarkov} if we show that any nonzero vector ${\bf v}\in\ML(\MA^{(r)})\setminus T$ has a proper strongly semiconformal decomposition. 

Without loss of generality we may assume that we are in the setting of Proposition~\ref{completeint}. Let ${\bf v}\in\ML(\MA^{(r)})\setminus T$ be a nonzero vector. We may assume via Remark~\ref{declaw} that the type of ${\bf v}$ is $r$.  By Lemma~\ref{sccompint} we know that for each row vector ${\bf v}_i$ of ${\bf v}$ either ${\bf v}_i=\alpha_i (-{\bf u}_1) +_{sc} \beta_i (\pm {\bf u}_2)$ or ${\bf v}_i=\beta_i (\pm {\bf u}_2) +_{sc} \alpha_i{\bf u}_1$
 for  $\alpha_i,\beta_i\in\NN$ with $\alpha_i+\beta_i>0$. We have three cases to analyze:
\begin{enumerate} 
\item[(a)]  $\alpha_i\beta_i\neq 0$ for all $i=1,\ldots,r$,  
\item[(b)] $\exists$   $i\in\{1,\ldots,r\}$ such that $\alpha_i=0$,
\item[(c)] $\exists$   $i\in\{1,\ldots,r\}$ such that $\beta_i=0$. 
\end{enumerate}

For case (a) we apply Lemma~\ref{kindtwoci} and we are done.  For case (b) let $i\in\{1,\ldots,r\}$ be such that $\alpha_i=0$, and since $\alpha_i+\beta_i>0$ then $\beta_i>0$. Since the type of ${\bf v}$ is $r$, by Lemma~\ref{sccompint},  ${\bf v}$ is of the form given in Lemma~\ref{kindtwoci}, the only difference being that some of the coefficients $\alpha_i,\beta_i$ might be zero, (not simultaneously). If $i\leq s$ then it follows from (\ref{need}) that there exists $j$ with $s+1\leq j\leq r$ such that $\beta_j\neq 0$. Let ${\bf z}$ be the type two vector with $i$-th row  ${\bf u}_2$ and $j$-th row  $-{\bf u}_2$ and let 
${\bf w}={\bf v}-{\bf z}$. Then ${\bf w}$ is nonzero since the type of ${\bf v}$ is $r\geq 3$. Moreover by the sign patterns it follows that ${\bf v}={\bf z}+_{sc}{\bf w}$. That the sum is also strongly semiconformal follows either by Lemma~\ref{sc=ssc} if $\beta_i=1$ or by noticing that ${\bf v}_i^+\neq{\bf z}_i^+$ and ${\bf v}_i^-\neq{\bf w}_i^-$ if $\beta_i>1$, which implies ${\bf v}^+\neq{\bf z}^+$ and ${\bf v}^-\neq{\bf w}^-$. For case (c) a similar argument as in case (b) holds. 
\end{proof}  

We note that in the case of a monomial curve $\MA$ in $\AA^3$ whose corresponding toric ideal is a complete intersection,  $\MS(\MA)$ consists of exactly one vector and the Graver basis of $\MS(\MA)$ is $\{\bf 0\}$. Thus in this case the bound given in \cite[Theorem 3.11]{HS} is $0$, and it is strictly smaller than $m(\MA)$.

\section{Graver complexity of monomial curves in $\AA^3$}

In \cite[Theorem 3]{SS},  it was shown that  $g(\MA)$, the Graver complexity of $\MA$,  is the maximum $1$-norm of any element in the Graver basis of the Graver basis of $\MA$. However computing the Graver basis of the Graver basis of $\MA$ is computationally challenging to say the least. In this section we give a lower bound for the Graver complexity of a  monomial curve $\MA$ in $\AA^3$. This shows that in general the upper bound for Markov complexity is rather crude: given any $k\in \NN$, one can find appropriate configuration $\MA=\{n_1,n_2,n_3\}$ so that the $g(\MA)\ge k$, while $m(\MA)\le 3$. First we show that in order to compute the Graver complexity of a monomial curve in $\AA^3$, one can assume that the configuration consists of vectors that are pairwise relatively prime. 

\begin{Proposition}\label{red_graver} Let $\MA=\{n_1,n_2,n_3\}$ such that $(n_1,n_2,n_3)=1$. For  $1\leq i< j\leq 3$ we let $d_{ij}=(n_i, n_j)$ and consider $ \MA_{red}= \{ {n_1}/{d_{12} d_{13}},\; {n_2}/{d_{12} d_{23}},\; {n_3}/{d_{13} d_{23}}\}$. Then  $g(\MA)=g(\MA_{red})$.
\end{Proposition}
\begin{proof}  Let  $d\in \NN$  divide both $n_1,n_2$ and consider $\MA'= \{ n_1/d, n_2/d, n_3\}$. It is clear that $(a_1,  a_2, a_3)\in \ML(\MA)$ if and only if $(a_1, a_2, a_3/d) \in \ML(\MA')$.  This implies  that there is a one-to-one correspondence between the elements of the Graver bases of $\MA$ and $\MA_{red}$. 
Moreover it is clear that if there is an element of type $t$ in the Graver basis of $\MA^{(r)}$ then there is an element of type $t$ in the Graver basis of $\MA_{red}^{(r)}$ and vice versa. Hence $g(\MA)=g(\MA_{red})$.
\end{proof}

To prove the lower bound of the next theorem, we create an  element whose type is the desired lower bound in the appropriate Lawrence lifting. 

\begin{Theorem}\label{lowerboundgraver}
Let $\MA=\{n_1,n_2,n_3\}$ such that $(n_1,n_2,n_3)=1$ and $d_{ij}=(n_i, n_j)$ for all $i\neq j$. Then 
\[
g(\MA)\geq \frac{n_1}{d_{12}d_{13}} + \frac{n_2}{d_{12}d_{23}} + \frac{n_3}{d_{13}d_{23}}.
\]
In particular, if $n_1,n_2,n_3$ are pairwise prime then $g(\MA)\geq n_1+n_2+n_3$. 
\end{Theorem}
\begin{proof} By Proposition~\ref{red_graver} $g(\MA)=g(\MA_{red})$ and thus we may assume that $d_{ij}=1$ for all $i\neq j$. We note that  ${\bf v}_1=(n_2, -n_1, 0)$, ${\bf v}_2=(n_3, 0, -n_1)$ and ${\bf v}_3=(0, n_3, -n_2)$ are clearly in the Graver basis of $\MA$. Let $k=n_1+n_2+n_3$ and consider   the $r\times n$ matrix ${\bf w}$ with row vectors ${\bf w}_i$ for $i=1,\ldots, k$ so that ${\bf w}_i={\bf v}_1$ for the first $n_1$ rows,  ${\bf w}_i=-{\bf v}_2$ for the next $n_2$ rows and  ${\bf w}_i={\bf v}_3$ for the last $n_3$ rows. It easy to see that ${\bf w}\in \ML(\MA^{(k)})$. We will show that $\bf w$ is in the Graver basis of $\MA^{(k)}$. Assume by contradiction that ${\bf w}\notin\MG(\MA^{(k)})$. Then ${\bf w}={\bf z}+_c{\bf u}$ for some nonzero vectors ${\bf z,u}\in\ML(\MA^{(k)})$. Since ${\bf w}_i\in\{{\bf v}_1,-{\bf v}_2,{\bf v}_3\}$ and  belongs to $\MG(\MA)$ it follows that ${\bf z}_i$ or ${\bf u}_i$ must equal ${\bf w}_i$ for $1\le i\le k$. Therefore by summing up the rows of ${\bf z}$ we obtain the relation 
\begin{equation}\label{eq_grav}
t_1{\bf v}_1-t_2{\bf v}_2+t_3{\bf v}_3={\bf 0}
\end{equation}
for some nonnegative integers $t_1,t_2,t_3$ such that $t_1+t_2+t_3<n_1+n_2+n_3$. However, Equation~\ref{eq_grav} implies immediately that $t_1n_2=t_2n_3$, $t_1n_1=t_3n_3$ and $t_2n_1=t_3n_2$ and since $n_1,n_2,n_3$ are pairwise relatively prime by assumption, it follows  that $n_3|t_1$, $n_2|t_2$ and $n_1|t_3$. Hence  $n_1+n_2+n_3\leq t_1+t_2+t_3$ a contradiction. 
\end{proof}

The following example shows that in general the inequality from Theorem~\ref{lowerboundgraver} can be strict.

\begin{Examples}{\rm (a) Let $\MA=\{3,4,5\}$. Computations with 4ti2 show  that the maximum  $1$-norm of the elements of $\MG(\MG(\MA))$ is $12$ and thus $g(\MA)$ equals the the lower bound of Theorem~\ref{lowerboundgraver}.

(b) Let $\MA=\{2,3,17\}$. Computations with 4ti2 show that the maximum  $1$-norm of the elements of $\MG(\MG(\MA))$ is $30$ and thus $g(\MA)=30$, while the lower bound of Theorem~\ref{lowerboundgraver} is $22$.}
\end{Examples} 

For the rest of this section we want to briefly comment about the Markov complexity of {\it generalized Lawrence liftings} of monomial curves in $\AA^3$. Given $\MA$, $r\ge 2$ and  $B\in \MM_{d\times 3}(\NN)$, the $r$-th generalized Lawrence lifting of $\MA$ with $B$, $\Lambda(\MA,B,r)$, differs from $\MA^{(r)}$ in the last row block: $B$  replaces $I_n$, see  \cite[Definition 3.3]{HS}. A vector ${\bf v}$ belongs to $\ML(\Lambda(\MA,B,r))$ if and only if its row vectors ${\bf v}_i\in\ML(\MA)$ for all $i=1,\ldots,r$, and $\sum_{i=1}^r {\bf v}_i\in\Ker(B)$. Therefore $\ML(\MA^{(r)})\subset\ML(\Lambda(\MA,B,r))$ for all $r\geq 2$. We define  $m(\MA,B)$ as  the largest type of any vector in the universal Markov basis of $\Lambda(\MA,B,r)$ as $r$ varies.  
 
\begin{Example}\label{genmarkov}
{\em Let $\MA=\{3,4,5\}$. It follows from Theorem~\ref{notcompint} that $I_{\MA}$ is not a complete intersection, and $\MM(\MA)$ consists of the following three vectors ${\bf u}_1=(-3,1,1)$, ${\bf u}_2=(1,-2,1)$ and ${\bf u}_3=(2,1,-2)$. Applying Theorem~\ref{markovnotci} we obtain $m(\MA,I_3)=m(\MA)=3$. 

Let $B_1=(1\ 3\ 0)$. One can show that $m(\MA,B_1)=2$. We briefly indicate how to prove this equality, without details, since the proof is similar to the ones given in Section 2. More precisely, with the same techniques from Lemma~\ref{indispensablemon} and Lemma~\ref{kindtwo}, and applying Proposition~\ref{universalmarkov}, one can prove that for any $r\geq 2$ the universal Markov basis $\MM(\Lambda(\MA,B,r))$ consists of the vectors of type $1$ with the nonzero row ${\bf u}_1$ and the vectors of type $2$ such that its two nonzero rows are ${\bf u}_2,{\bf u}_3$ or ${\bf u},-{\bf u}$, with $u\in\MG(\MA)\setminus\{{\bf u}_1\}$. We also note that in this case $m(\MA,B_1)$ equals the lower bound given by Ho\c sten and Sullivant in \cite[Theorem 3.11]{HS}.

However if $B_2=(1\ n\ 0)$, and $n\geq 4$ then by \cite[Theorem 3.11]{HS} it follows that $m(\MA,B_2)\geq \;{(3n+1)}/{5}$ if $n\equiv 3(\mod 5)$ or 
$m(\MA,B_2)\geq 3n+1$ otherwise. Indeed, note that $B_2\cdot \MS(\MA)=(n-3\ 1-2n\ n+2)$ and in the Graver basis of this matrix there is the circuit $(0,\frac{n+2}{5},\frac{2n-1}{5})$ if $n\equiv 3(\mod 5)$, and $(0,n+2,2n-1)$ otherwise.}
\end{Example} 
 
The previous example shows that while the Markov complexity $m(\MA)$ is bounded above by three for any monomial curve in $\AA^3$, the Markov complexity  $m(\MA,B)$ of the generalized Lawrence liftings of monomial curves in $\AA^3$ is not bounded. Based on computations with 4ti2 \cite{4ti2}, we are tempted to raise the following question. 

\begin{Question}\label{one} {\em Let $\MA$ be a monomial curve in $\AA^3$ so that $I_{\MA}$ is not a complete intersection.  Is $m(\MA,B)$  equal to the maximum $1$-norm of the elements in the Graver basis of $B\cdot \MS(\MA)$, i.e. the lower bound given   in \cite[Theorem 3.11]{HS}?} 
\end{Question}   
 
Computing  Markov complexity is an extremely challenging problem, and a formula for it seems hard to find in general. We pose a final question  based on extensive computational evidence.

\begin{Question}\label{two} 
{\em Let $A\in\MM_{m\times n}(\NN)$. Is   $m(A)$  equal to the smallest integer $r$ such that a minimal Markov basis  of $A^{(r+1)}$ does not contain elements of type $r+1$?}    
\end{Question}

\end{document}